\documentclass[11pt,a4paper]{amsart}
\usepackage[utf8]{inputenc}
\usepackage[english]{babel}

\usepackage[T1]{fontenc}
\usepackage{url}

\usepackage{mathrsfs}
\usepackage{tikz}

\title[On the question of genericity of hyperbolic knots]{On the question of genericity \\of hyperbolic knots}
\author{Andrei Malyutin}
\address{%
St.\,Petersburg Department of 
Steklov Institute of Mathematics
}
\email{malyutin@pdmi.ras.ru}

\def \mspace {\!}
\def \crn {\operatorname{cr}}

\def \clos {\operatorname{clos}}
\def \card {\operatorname{card}}
\def \dist {\operatorname{dist}}
\def \span {\operatorname{span}}
\newcommand \sm {\setminus}
\newcommand \R      {\mathbb R}
\newcommand \dd      {\partial}
\newcommand \inr      {\operatorname{int}}
\newcommand \bn      {\operatorname{b}}

\newcommand \be     {\begin{equation}}
\newcommand \ee     {\end{equation}}

\newtheorem{thm}{Theorem}
\newtheorem{prop}{Proposition}
\newtheorem{lem}{Lemma}

\newtheorem{claim}{Claim}
\newtheorem{conj}{Conjecture}
\newtheorem{cor}{Corollary}

\begin{document}
\maketitle

\begin{abstract}
A well-known conjecture in knot theory says that the percentage of hyperbolic knots amongst all of the prime knots of~$n$ or fewer crossings approaches $100$ as $n$ approaches infinity.
In this paper, it is proved that this conjecture contradicts several other plausible conjectures, including the 120-year-old conjecture
on additivity of the crossing number of knots under connected sum and the conjecture that the crossing number of a satellite knot is not less than that of its companion.
\end{abstract} 

\maketitle

\section{Introduction}

William Thurston proved in 1978 that every non-torus non-satellite knot is a hyperbolic knot.
Computations show that the overwhelming majority of prime knots with small crossing number are hyperbolic knots.
The following table gives the number of hyperbolic, prime satellite, and torus knots of $n$ crossings for $n=3,\dots,16$
(see~\cite{HTW98} or the sequences A002863, A052408, A051765, and A051764 in the Sloane's encyclopedia of integer sequences).

\footnotesize
\begin{table}[ht]
	\begin{center}
		\begin{tabular}{ |c|c|c|c|c|c|c|c|c|c|c|c|c|c|c|c|c|c|c|c|c|c| } 
			 \hline
			 type~~$\backslash$~~\vphantom{$L^{L^L}$}$n=$  &  3 & 4 & 5 & 6 & 7 & 8 & 9 & 10 & 11 & 12 & 13 & 14 & 15 & 16 \\ 
			 \hline
			  all prime \vphantom{$L^{L^L}$} & 1 & 1 & 2 & 3 & 7 & \mspace21\mspace & \mspace49\mspace & \mspace165\mspace & \mspace552\mspace &  \mspace2\,176\mspace &  \mspace9\,988\mspace & \mspace46\,972\mspace & \mspace253\,293\mspace & \mspace1\,388\,705 \\ 
			 hyperbolic & 0 & 1 & 1 & 3 & 6 & \mspace20\mspace & \mspace48\mspace & \mspace164\mspace & \mspace551\mspace &  \mspace2\,176\mspace &  \mspace9\,985\mspace & \mspace46\,969\mspace & \mspace253\,285\mspace & \mspace1\,388\,694 \\ 
			 prime satellite &  0 & 0 & 0 & 0 & 0 & 0 & 0 & 0 & 0 & 0 & 2 & 2 & 6 & 10 \\
			 torus     &  1 & 0 & 1 & 0 & 1 & 1 & 1 & 1 & 1 & 0 & 1 & 1 & 2 & 1 \\ 
			 \hline
		\end{tabular}
		\vspace{0pt}
		\caption{Number of prime knots}
		\label{tbl:table}
	\end{center}
\end{table}
\normalsize
\vspace{-15pt}

(A~part of) these data gave rise to the following conjecture (see~\cite[p.\,119]{Ad94b}).

\begin{conj}\label{H1}
The percentage of hyperbolic knots amongst all of the prime knots of~$n$ or fewer crossings approaches $100$ as $n$ approaches infinity.
\end{conj}
 
In the present paper, we show that Conjecture~\ref{H1} contradicts several other long standing conjectures, including the following one.

\begin{conj}\label{con:add}
The crossing number of knots is additive with respect to connected sum.
\end{conj}

See, e.\,g.,~\cite[p.\,69]{Ad94b}, \cite[Problem 1.65]{Kir97}, and~\cite{La09} for comments and related results. 
Another related conjecture is as follows.

\begin{conj}\label{con:sat}
The crossing number of a satellite knot is bigger (a weaker variant: not less) than that of its companion.
\end{conj}

See~\cite[p.\,118]{Ad94b}, \cite[Problem 1.67 (attributed to de Souza)]{Kir97}, and~\cite{La14}.
It is remarked in \cite[Problem 1.67]{Kir97} concerning Conjecture~\ref{con:sat} that `Surely the answer is yes, so the problem indicates the difficulties of proving statements about the crossing number'.
Since a composite knot is a connected sum of its factors and, at the same time, 
is a satellite of each of its factors, 
the `intersection' of Conjectures~\ref{con:add} and~\ref{con:sat} yields the following.

\begin{conj}\label{con:com}
The crossing number of a composite knot is bigger (a weaker variant: not less) than that of each of its factors.
\end{conj}

Let us denote by $\crn(X)$ the crossing number of a knot~$X$.
If $P$ is a prime knot and $\lambda$ is a real number,
we say that $P$ is \emph{$\lambda$-regular} if we have $\crn(K)\ge \lambda\cdot\crn(P)$ 
whenever $P$ is a factor of a knot $K$.
In this terminology, Conjecture~\ref{con:com} says that each prime knot is $1$-regular.
Lackenby~\cite{La09} proved that each knot is $\frac1{152}$-regular.
Our considerations involve the following conjecture.

\begin{conj}\label{con:2/3}
Each prime knot is $\frac23$-regular.
\end{conj}

We also consider the following weakening of Conjecture~\ref{con:2/3}.

\begin{conj}\label{conj-weak}
There exist $\varepsilon>0$ and $N>0$ such that, for all $n>N$, the percentage of $\frac23$-regular knots amongst all of the hyperbolic knots of~$n$ or fewer crossings is at least~$\varepsilon$.
%
%
\end{conj}

We have the following obvious implications.
\begin{center}
\begin{tikzpicture} 
\draw (0,0) node[anchor=west]{Conj.~4 $~\Longrightarrow~$ Conj.~5 $~\Longrightarrow~$ Conj.~6};
\draw (120:1.2cm) node [left]{Conj.~2};
\draw (120:0.6cm) node [rotate=-45]{$\Longrightarrow$};
\draw (-120:1.2cm) node [left]{Conj.~3};
\draw (-120:0.6cm) node [rotate=45]{$\Longrightarrow$};
\end{tikzpicture} 
\end{center}

The main result of this paper is the following theorem.

\begin{thm}\label{T1}
Conjecture~\ref{H1} contradicts (each of) Conjectures
\ref{con:add},
\ref{con:sat}, 
\ref{con:com},
\ref{con:2/3},  and
\ref{conj-weak}.
\end{thm}

The paper is organised as follows. 
Section~\ref{sec:rems} contains remarks concerning Conjectures~\ref{H1}--\ref{conj-weak}.
In Section~\ref{sec:idea}, we present the key idea of the proof of Theorem~\ref{T1} and reduce Theorem~\ref{T1} to Proposition~\ref{lem:main} consisting of three assertions.
Sections~\ref{sec:proof-i}--\ref{sec:proof-iii} contain the proof of Proposition~\ref{lem:main}.
In Sections~\ref{sec:proof-i} and~\ref{sec:proof-ii}, we prove the first two assertions of Proposition~\ref{lem:main}. 
Section~\ref{sec:combinatorial_lemma} 
contains a combinatorial lemma used in the proof of the last assertion of Proposition~\ref{lem:main}.
Section~\ref{sec:tangles} contains preliminaries on tangles.
In Section~\ref{sec:proof-iii}, we prove the last assertion of Proposition~\ref{lem:main}. 
In Section~\ref{sec:add-I}, we introduce a new property of knots (strong property PT) 
and prove Theorem~\ref{T2} strengthening Theorem~\ref{T1}.
In Section~\ref{sec:add-II}, we show that 
an assumption that Conjectures~\ref{con:add}--\ref{con:2/3} has 
many strong counterexamples contradicts Conjecture~\ref{H1} as well.
In Section~\ref{sec:add-III}, we show that certain assumptions concerning unknotting numbers of knots contradict Conjecture~\ref{H1}.


The paper should be interpreted as being in either the PL or smooth category.
For standard definitions we mostly use the conventions of \cite{BZ06} and~\cite{BZH14}.
There will be a certain abuse of language in order to avoid complicating the notation. 
In particular, a \emph{knot} $K$ will be a circle embedded in a $3$-sphere $S^3$, 
a pair $(S^3, K)$, or a class of homeomorhic pairs (cf.~\cite[p.~1]{BZ06}).
No orientations on knots and spaces are placed if not otherwise stated.

The author is grateful to 
Ivan Dynnikov,
Evgeny Fominykh,
Aleksandr Gaifullin,
Vadim Kaimanovich,
Maksim Karev,
Paul Kirk,
Vladimir Nezhinskij,
Sem\"{e}n Podkorytov,
J\'{o}zef Przytycki, 
Alexey Sleptsov, and
Andrei Vesnin
for helpful comments and suggestions.

\section{Remarks} 
\label{sec:rems}

We list certain results related to Conjectures \ref{H1}--\ref{conj-weak}.


\subsection*{Predominance of hyperbolic objects} 
In recent years, a number of results have been obtained showing the predominance of hyperbolic objects in various cases.
We refer to the works of Ma~\cite{Ma14} and Ito~\cite[Theorem\,2]{Ito15} for results concerning genericity of hyperbolic knots and links.
See also~\cite{Mah10a}, \cite[Theorem\,2]{LMW14}, \cite{LusMo12}, \cite{Riv14}, and~\cite[Theorem\,1]{Ito15} for results on genericity of hyperbolic 3-manifolds.
Related results show that pseudo-Anosovs prevail (in various senses) in mapping class groups of surfaces.
We refer to 
\cite{Riv08, Riv09, Riv10, Riv12, Riv14}, \cite{Kow08},
\cite{Mah10b, Mah11, Mah12}, 
\cite{AK10}, \cite{Sis11}, \cite{Mal12}, \cite{LubMe12}, \cite{MS13}, and also \cite{Car13, CW13, Wi14} (non-random approach) for precise statements and detailed discussions.
See also~\cite{GTT16} for the genericity of loxodromic isometries for actions of hyperbolic groups on hyperbolic spaces.
Other examples of hyperbolicity predominance can be found in extensive literature on exceptional Dehn fillings (see \cite{Thu79} etc.) and in~\cite[0.2.A]{Gro87}, \cite[p.\,20]{GhH90}, \cite{Ch91} and \cite{Ch95}, \cite{Ols92}, \cite{Gro93}, \cite{Zuk03}, \cite{Oll04}, \cite{Oll05}, 
where viewpoints are given from which it appears that a generic finitely presented group is word hyperbolic. 
Apparently, combining approaches developed by~\cite{Ito15} and~\cite{Ma14} with results of \cite{Car13, CW13, Wi14} (the same for \cite{LusMo12}) one can obtain more viewpoints where generic knot will be hyperbolic.

\subsection*{Predominance of non-hyperbolic objects} 
As for natural models where it is proved that hyperbolic objects are rare,
we have the standard methods of generating knots as polygons in~${\mathbb R}^3$.
Under this approach, composite knots prevail and prime knots (including hyperbolic ones) are asymptotically scarce.
See Sumners and Whittington~\cite{SW88}, Pippenger~\cite{Pip89}, and also Soteros, Sumners, and Whittington~\cite{SSW92} 
for the case of self-avoiding random polygons on the simple cubic lattice; 
see \cite{Oetal94} and \cite{Sot98} for such polygons in specific subsets of the lattice; 
see \cite{DPS94} and \cite{Jun94} for local and global knotting in Gaussian random polygons;
see \cite{Di95} and \cite{DNS01} for local and global knotting in equilateral random polygons;
see also \cite{Ken79} for knotting of Brownian motion and  \cite{Sum09} and \cite{MMO11} for more references.
An interesting idea has appeared in \cite[p.\,4]{Ad05} and \cite[p.\,95]{Cr04} that prime satellite knots should prevail over hyperbolic ones when we consider 
Gaussian random polygons. 
In both \cite{Ad05} and \cite{Cr04}, however, the idea was apparently inspired by a misinterpretation of results in~\cite{Jun94}.

\subsection*{Crossing number additivity} 
Murasugi~\cite[Corollary~6]{Mur87} proved that Conjecture~\ref{con:add} is valid for alternating knots. (This follows from the proof of Tait conjecture that reduced alternating projections are minimal;
this Tait conjecture was also proved, independently, by Kauffman~\cite{Kau87} and Thistlethwaite~\cite{Thi87}.) Conjecture~\ref{con:add} is valid for \emph{adequate} knots (see~\cite{LT88}).
Diao~\cite{Di04} and Gruber~\cite{Gru03} independently proved that Conjecture~\ref{con:add} is valid for torus knots and certain other particular classes of knots.
Results of \cite{Mur87}, \cite{Kau87}, \cite{Thi87} imply that alternating knots are 1-regular.\footnote{It is shown in \cite{Mur87}, \cite{Kau87}, \cite{Thi87} that 
(i) for each knot $K$ we have $\span V_K(t)\le\crn(K)$, and
(ii) for alternating $K$ we have $\span V_K(t)=\crn(K)$,
where $V_K(t)$ is the Jones polynomial of~$K$ and $\span V_K(t)$ denotes the difference between the maximal and minimal degrees of $V_K(t)$. (It is known that $V_K(t)\neq 0$ so that $\span V_K(t)$ is well-defined. See~\cite[Theorem~15]{Jon85}.) Then 1-regularity of alternating knots follows because, for any knots $K_1$ and $K_2$, we have \cite[Theorem~6]{Jon85}
$$
\span V_{K_1\sharp K_2}(t)=\span V_{K_1}(t)+\span V_{K_1}(t).
$$}
Diao~\cite[Theorem 3.8]{Di04} showed that torus knots 
are 1-regular. 
In \cite{PZ15}, the authors 
introduce a telescopic family of conjectures concerning monotonic simplification of link diagrams and provide supporting evidence for (the strongest of) these conjectures.\footnote{Ivan Dynnikov (private communication) found a counterexample to Petronio--Zanellati conjectures.}
Each of Petronio--Zanellati conjectures implies Conjecture~\ref{con:add}.

\subsection*{Torus knots} 
Murasugi~\cite[Proposition~7.5]{Mur91} proved that the torus link of type $(p, q)$, where $2\le p\le q$, has crossing number $(p-1)q$.
Taking into account that the number of all prime knots of~$n$ crossings grows exponentially in~$n$
(see~\cite{ES87,Wel92}), this implies that 
the percentage of torus knots amongst all of the prime knots of~$n$ or fewer crossings approaches $0$ as $n$ approaches infinity.
Thus, only satellite knots pose a danger to Conjecture~\ref{H1}.

\subsection*{Hyperbolic knots} 
Several interesting classes of knots are known to consist of hyperbolic and torus knots only.
In particular, amongst these classes are:

-- prime alternating knots, including 2-bridge knots (see \cite{Men84}),

-- prime almost alternating knots (see \cite{Aetal92}), 

-- prime toroidally alternating knots (see \cite{Ad94a}), 

-- arborescent knots, including 2-bridge knots, pretzel knots, and Montesinos knots (see \cite{BS10}, 
Theorem~1.5 and subsequent discussion in \cite{FG09}), etc.

More families of hyperbolic knots, links, and tangles are listed in~\cite{Ad05}.
See \cite{Ito11}, \cite[Theorems 8.3, 8.4]{IK12} for new examples of huge classes of hyberbolic knots, links, and $3$-manifolds.

\section{The idea of the proof of Theorem~\ref{T1}}
\label{sec:idea}

Our proof of Theorem~\ref{T1} uses a specific way of constructing satellite knots. 
For brevity, we use the term \emph{$\gamma$-knots} for the satellite knots constructed in this way.

\subsection*{Definition.~$\gamma$-Knots.}
Let $K$ be a knot in a $3$-sphere $S^3$, and let $V$ be an unknotted solid torus in~$S^3$ such that $K$ is contained in the interior of~$V$. 
Let $\psi\colon V\to W\subset S^3$ be a homeomorphism onto a tubular neighbourhood~$W$ of a hyperbolic knot.
Recall that the \emph{winding number} of~$K$ in~$V$ is the absolute value of the algebraic intersection number of~$K$ with a meridional disk in~$V$.
Assume that the winding number of~$K$ in~$V$ is at least~$2$
and that $\psi$ maps a longitude\footnote{If a solid torus $U$ is embedded in a $3$-sphere~$S^3$, 
then there exists an essential curve in~$\dd U$ that bounds a $2$-sided surface in $S^3 \sm \inr(U)$ (a Seifert surface). 
This curve is unique up to isotopy on~$\dd U$ and is called a \emph{longitude} of $U$ in~$S^3$ (see, e.\,g., \cite[Theorem~3.1]{BZ06}).} of~$V$ to a longitude of~$W$.
Then we say that the knot $\psi(K)\subset S^3$ is a \emph{$\gamma$-knot over}~$K$.

A~method of constructing a $\gamma$-knot is given in Fig.~\ref{fig:gamma}.
Assume that a diagram~$D'$ of a knot~$K'$ is obtained from a diagram~$D$ of a knot~$K$ by local move as in Fig.~\ref{fig:gamma}.
(See Fig.~\ref{fig:gamma-tref} for an example.)
Our definitions imply that if two arrows on arcs in Fig.~\ref{fig:gamma}(a) indicate the same orientation on~$K$, then $K'$ is a $\gamma$-knot over~$K$. 
Here, the winding number is~$2$ while the companion hyperbolic knot
is the figure-eight knot.
(In order to check that the condition on longitudes is also fulfilled, we observe that each arc in Fig.~\ref{fig:gamma} has zero total curvature.)

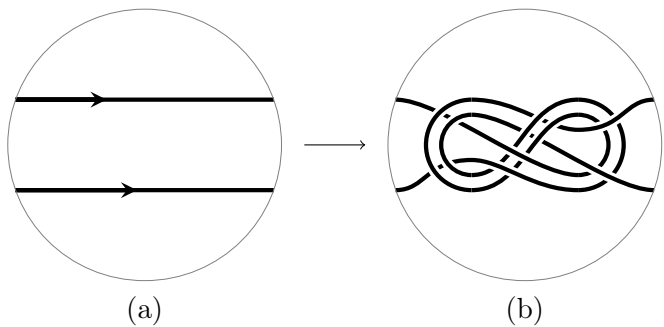
\begin{figure}[ht]
\begin{center}

\begin{tikzpicture} 
[knots/.style={white,double=black,very thick,double distance=1.6pt}]
\draw[line width=1.6pt]  [-stealth]
(-1/5,3/5) -- (16/5,3/5) 
(-1/5,3/5) -- (5/5,3/5);
\draw[line width=1.6pt] [-stealth]
(-1/5,-3/5) -- (16/5,-3/5)
(-1/5,-3/5) -- (7/5,-3/5);
 \draw[help lines](7.5/5,0) circle (9/5);
 \draw (7.5/5,-11/5)  node {(a)};

\draw[->] (20/5-0.4,0)--(20/5+.4,0);

\begin{scope}
[xshift=5cm] 

\draw[knots]
(11/5,-2/5) 
.. controls +(1.111/5,0) and +(0,-1.111/5) .. ++(2/5,2/5)
.. controls +(0,1.111/5) and +(1.111/5,0) .. ++(-2/5,2/5)
(11/5,-3/5) 
.. controls +(1.666/5,0) and +(0,-1.666/5) .. ++(3/5,3/5)
 .. controls +(0,1.666/5) and +(1.666/5,0) .. ++(-3/5,3/5);
\draw[knots] 
(4/5,3/5)  .. controls +(3/5,0) and +(-2/5,-.0/5) .. (11/5,1/5) .. controls +(3/5,0) and +(-2/5,-.0/5) .. (16/5,3/5)
(4/5,2/5) .. controls +(3/5,0) and  +(-3/5,-.0/5) .. (16/5,-3/5);
\draw[knots] 
(11/5,2/5) .. controls +(-3/5,0) and +(3.5/5,0) .. (4/5,-3/5)
(11/5,3/5) .. controls +(-3.5/5,0) and +(3/5,0) .. (4/5,-2/5);
%
\draw[knots] 
(-1/5,3/5) .. controls +(3/5,0) and  +(-3/5,0) ..  (11/5,-2/5)
(-1/5,-3/5) .. controls +(2/5,0) and +(-3/5,0) .. (4/5,-1/5)  .. controls +(2/5,0) and +(-3/5,0) .. (11/5,-3/5);
%
\draw[knots] 
 (4/5,-3/5)
 .. controls +(-1.666/5,0) and +(0,-1.666/5) .. ++(-3/5,3/5)
 .. controls +(0,1.666/5) and +(-1.666/5,0) .. ++(3/5,3/5)
(4/5,-2/5)
 .. controls +(-1.111/5,0) and +(0,-1.111/5) .. ++(-2/5,2/5)
 .. controls +(0,1.111/5) and +(-1.111/5,0) .. ++(2/5,2/5);
 \draw[help lines](7.5/5,0) circle (9/5);
 \draw (7.5/5,-11/5)  node {(b)};
 \end{scope}
\end{tikzpicture} 
\end{center}
\caption{Double figure-eight move} \label{fig:gamma}
\end{figure}

We deduce Theorem~\ref{T1} from the following proposition on $\gamma$-knots.

\begin{prop}\label{lem:main}
{\rm(i)} Each $\gamma$-knot is a satellite knot. 

{\rm(ii)} The sets of $\gamma$-knots over distinct non-satellite knots are disjoint.

{\rm(iii)} If $P$ is a $\frac23$-regular prime knot, then there exists a \underline{prime} $\gamma$-knot $P'$ over~$P$ with $\crn(P')\le \crn(P)+17$.
\end{prop}

\subsection*{Remark}
Assertion~(iii) of Proposition~\ref{lem:main}
is not obvious because a $\gamma$-knot over a prime knot is not necessarily prime (see Fig.~\ref{fig:gamma-tref}).

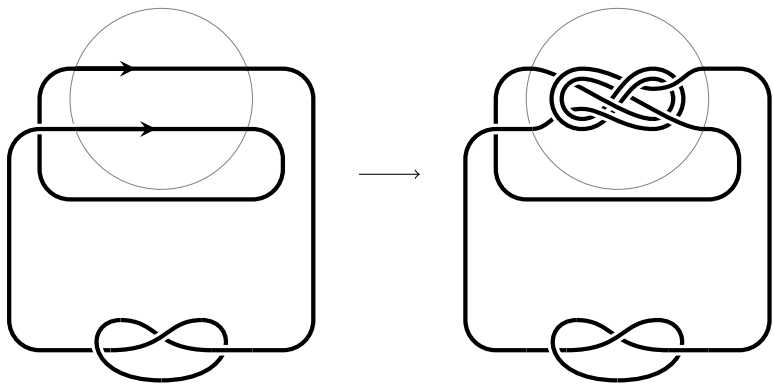
\begin{figure}[ht]
\begin{center}

\begin{tikzpicture} 
[knots/.style={white,double=black,very thick,double distance=1.6pt}]

\begin{scope}
[scale=2/3]

\draw[line width=1.6pt]  [-stealth]
(-1/5,3/5) -- (16/5,3/5) 
(-1/5,3/5) -- (5/5,3/5);
\draw[line width=1.6pt] [-stealth]
(-1/5,-3/5) -- (16/5,-3/5)
(-1/5,-3/5) -- (7/5,-3/5);
 \draw[help lines](7.5/5,0) circle (9/5);

\begin{scope}
[xshift=1.5cm, yshift=-5cm]

\draw[knots] (-9/5,0/5)--(-5/5,0/5);
\draw[knots] (4/5,3/5) .. controls +(3.666/5,0) and +(7.666/5,0) .. (0/5,-3/5);
\draw[knots] (0/5,-3/5)  .. controls +(-7.666/5,0) and +(-3.666/5,0) .. (-4/5,3/5);
\draw[knots] (-4/5,3/5)  .. controls +(3.666/5,0) and +(-5.666/5,0) .. (5/5,0/5);
\draw[knots] (9/5,0/5)--(5/5,0/5);
\draw[knots] (-5/5,0/5) .. controls +(5.666/5,0) and +(-3.666/5,0) .. (4/5,3/5);

\draw[line width=1.6pt] (9/5,0/5)--(12/5,0/5) arc (-90:0:3/5)--(15/5,25/5) arc (0:90:3/5)--(8.4/5,28/5);
\draw[line width=1.6pt] (-8.4/5,28/5)--(-9/5,28/5) arc (90:180:3/5)--(-12/5,18/5) arc (180:270:3/5)
--(9/5,15/5) arc (-90:0:3/5)--(12/5,19/5) arc (0:90:3/5)--(8.4/5,22/5);
\draw[knots]  (-8.4/5,22/5)--(-12/5,22/5) arc (90:180:3/5)--(-15/5,3/5) arc (180:270:3/5)
--(-9/5,0/5);
\end{scope}

\end{scope}

\draw[->] (20/5-0.4,-1)--(20/5+.4,-1);

\begin{scope}
[xshift=6cm, scale=2/3]

\draw[knots]
(11/5,-2/5) 
.. controls +(1.111/5,0) and +(0,-1.111/5) .. ++(2/5,2/5)
.. controls +(0,1.111/5) and +(1.111/5,0) .. ++(-2/5,2/5)
(11/5,-3/5) 
.. controls +(1.666/5,0) and +(0,-1.666/5) .. ++(3/5,3/5)
 .. controls +(0,1.666/5) and +(1.666/5,0) .. ++(-3/5,3/5);
\draw[knots] 
(4/5,3/5)  .. controls +(3/5,0) and +(-2/5,-.0/5) .. (11/5,1/5) .. controls +(3/5,0) and +(-2/5,-.0/5) .. (16/5,3/5)
(4/5,2/5) .. controls +(3/5,0) and  +(-3/5,-.0/5) .. (16/5,-3/5);
\draw[knots] 
(11/5,2/5) .. controls +(-3/5,0) and +(3.5/5,0) .. (4/5,-3/5)
(11/5,3/5) .. controls +(-3.5/5,0) and +(3/5,0) .. (4/5,-2/5);
%
\draw[knots] 
(-1/5,3/5) .. controls +(3/5,0) and  +(-3/5,0) ..  (11/5,-2/5)
(-1/5,-3/5) .. controls +(2/5,0) and +(-3/5,0) .. (4/5,-1/5)  .. controls +(2/5,0) and +(-3/5,0) .. (11/5,-3/5);
%
\draw[knots] 
 (4/5,-3/5)
 .. controls +(-1.666/5,0) and +(0,-1.666/5) .. ++(-3/5,3/5)
 .. controls +(0,1.666/5) and +(-1.666/5,0) .. ++(3/5,3/5)
(4/5,-2/5)
 .. controls +(-1.111/5,0) and +(0,-1.111/5) .. ++(-2/5,2/5)
 .. controls +(0,1.111/5) and +(-1.111/5,0) .. ++(2/5,2/5);
 \draw[help lines](7.5/5,0) circle (9/5);
 
\begin{scope}
[xshift=1.5cm, yshift=-5cm]

\draw[knots] (-9/5,0/5)--(-5/5,0/5);
\draw[knots] (4/5,3/5) .. controls +(3.666/5,0) and +(7.666/5,0) .. (0/5,-3/5);
\draw[knots] (0/5,-3/5)  .. controls +(-7.666/5,0) and +(-3.666/5,0) .. (-4/5,3/5);
\draw[knots] (-4/5,3/5)  .. controls +(3.666/5,0) and +(-5.666/5,0) .. (5/5,0/5);
\draw[knots] (9/5,0/5)--(5/5,0/5);
\draw[knots] (-5/5,0/5) .. controls +(5.666/5,0) and +(-3.666/5,0) .. (4/5,3/5);

\draw[line width=1.6pt] (9/5,0/5)--(12/5,0/5) arc (-90:0:3/5)--(15/5,25/5) arc (0:90:3/5)--(8.4/5,28/5);
\draw[line width=1.6pt] (-8.4/5,28/5)--(-9/5,28/5) arc (90:180:3/5)--(-12/5,18/5) arc (180:270:3/5)
--(9/5,15/5) arc (-90:0:3/5)--(12/5,19/5) arc (0:90:3/5)--(8.4/5,22/5);
\draw[knots]  (-8.4/5,22/5)--(-12/5,22/5) arc (90:180:3/5)--(-15/5,3/5) arc (180:270:3/5)
--(-9/5,0/5);
\end{scope}
\end{scope}
 
\end{tikzpicture} 
\end{center}
\caption{A composite $\gamma$-knot over the trefoil} \label{fig:gamma-tref}
\end{figure}

\label{sec:L=>T}

\begin{proof}[Proposition~{\rm\ref{lem:main}} implies Theorem~{\rm\ref{T1}}]
We introduce the following notation.
Let $p_n$ (resp., $h_n$, $s_n$) denote the number of prime (resp., hyperbolic, prime satellite) knots with crossing number~$n$.
We set $P_n=\sum_{k=1}^{n}p_n$, $H_n=\sum_{k=1}^{n}h_n$, and $S_n=\sum_{k=1}^{n}s_n$.

Since each of Conjectures 
\ref{con:add},
\ref{con:sat}, 
\ref{con:com}, and
\ref{con:2/3} 
implies 
Conjecture~\ref{conj-weak} (see the diagram before Theorem~\ref{T1}), 
it suffices to prove only that Conjectures~\ref{conj-weak} and~\ref{H1} are incompatible.
If Conjecture~\ref{conj-weak} is true, then there exist $\varepsilon_0>0$ and $N_0>0$ such that, for all $n>N_0$, the number of $\frac23$-regular hyperbolic knots of~$n$ or fewer crossings is at least~$\varepsilon_0 H_n$.
Obviously, in this case assertions~(i), (ii), and~(iii) of Proposition~\ref{lem:main} imply that (for all $n>N_0$) we have 
$$
S_{n+17}~\ge~ \varepsilon_0H_n.
$$
Therefore, we have 
$$P_{n+17}~\ge~ H_{n+17}+\varepsilon_0 H_n.$$
This is equivalent to the following inequality
\begin{equation}
\label{eq:PHPH}
1~\ge~ \frac{H_{n+17}}{P_{n+17}}+\varepsilon_0 \frac{H_n}{P_n}\frac{P_n}{P_{n+17}}.
\end{equation}
If Conjecture~\ref{H1} is true, then both sequences $\frac{H_{n+17}}{P_{n+17}}$ and $\frac{H_n}{P_n}$ tend to~$1$. 
In this case, Eq.~\eqref{eq:PHPH} implies that 
$$
\frac{P_{n+17}}{P_n}\xrightarrow{n\to +\infty} +\infty.
$$
Consequently, for each $B>0$ we have $P_{n}> B^n$ for all sufficiently large~$n$. (We consider subsequences of the form $P_{n_0+17i}$, $i\in\mathbb{N}$.)
In other words, we have
\begin{equation}
\label{eq:Pinfty}
P_{n}^{1/n} \xrightarrow{n\to +\infty} +\infty.
\end{equation}
However, it is shown in~\cite{Wel92} that
$$
\limsup_{k\to\infty} p_{n}^{1/n} < +\infty,
$$
which implies that there exists $B>0$ such that 
$p_n<B^n$ for all $n\in\mathbb{N}$.
Then, for each $n\in\mathbb{N}$ we have $P_n<(B+1)^n$ whence it follows that 
$$
\limsup_{k\to\infty} P_{n}^{1/n} \le B+1 < +\infty.
$$
This contradicts~\eqref{eq:Pinfty}. The obtained contradiction completes the proof.
\end{proof}

%

\section{Proof of assertion $\rm(i)$ of Proposition~\ref{lem:main}}
\label{sec:proof-i}


We recall definitions of satellite knots.
A knot $K$ in $S^3$ is a \emph{satellite knot} if $S^3$ contains a non-trivial knot~$C$ such that 
$K$ lies in the interior of a regular neighbourhood~$V$ of~$C$, $V$ does not contain a $3$-ball containing~$K$, and $K$ is not a core curve of the solid torus~$V$. 
The knot $K$ is a satellite knot if and only if $K$ contains an incompressible, non-boundary parallel torus in its complement. (For a proof, see~\cite[Remark~16.1, p.~335]{BZH14}.)

Let $K$ be a $\gamma$-knot in~$S^3$.
Then the definition of $\gamma$-knots implies that $K$ lies in a knotted solid torus $W\subset S^3$ such that the winding number of $K$ in~$W$ is at least~$2$.
Since the winding number of $K$ in $W$ is at least~$2$, 
it follows that $W$ does not contain a $3$-ball containing~$K$, and $K$ is not a core curve of~$V$.
This means by the above definition that $K$ is a satellite knot.

\section{Proof of assertion $\rm(ii)$ of Proposition~\ref{lem:main}}
\label{sec:proof-ii}

We show that the sets of $\gamma$-knots over distinct non-satellite knots are disjoint.
Suppose to the contrary that there exist a knot $K$ and two distinct non-satellite knots $H_1$ and $H_2$ such that $K$ is a $\gamma$-knot both over $H_1$ and over~$H_2$. 
%
%
By the definition of $\gamma$-knots, this means that there exist embedded solid tori $V_1$ and $V_2$ in $S^3$ and re-embeddings $\phi_1\colon V_1\to S^3$ and $\phi_2\colon V_2\to S^3$ such that, for each $i\in\{1,2\}$, the following conditions hold:

-- $V_i$ is a tubular neighbourhood of a hyperbolic knot,

-- $K$ lies in the interior of $V_i$ and the winding number of $K$ in $V_i$ is at least $2$,

-- the solid torus $\phi_i(V_i)$ is unknotted,

-- $\phi_i$ maps a longitude of $V_i$ to a longitude of $\phi_i(V_i)$,


-- we have $\phi_i(K)=H_i$.


\begin{claim}
\label{cl:Vi-incompressible}
The tori $\dd V_1$ and $\dd V_2$ are both incompressible in~$S^3\sm K$.
\end{claim}

Since the winding number of $K$ in $V_i$ is non-zero, it follows that no $3$-ball in $V_i$ contains~$K$.
If a knotted solid torus~$U$ in a $3$-sphere $S^3$ contains a knot $L$ in its interior while
no $3$-ball in $U$ contains~$L$, then $\dd U$ is sometimes called a \emph{companion torus} of~$L$. It is well known that, in this case, $\dd U$ is incompressible in~$S^3\sm L$. 
(See, e.\,g., \cite[Propositions~3.10 and~3.12, and E~2.9]{BZH14}.)
This implies Claim~\ref{cl:Vi-incompressible}.


\begin{claim}
\label{cl:iso-dd}
There exists an isotopy of $\dd V_1$ in $S^3\sm K$ that moves $\dd V_1$ to a position where $\dd V_1 \cap \dd V_2=\emptyset$.
\end{claim}

It may be assumed that $\dd V_1$ intersects $\dd V_2$ transversely in simple closed curves. 
If the intersection $\dd V_1\cap \dd V_2$ contains a curve that is inessential in $\dd V_2$, 
let $C$ be an innermost of such curves and let $d$ be the open disk in $\dd V_2\sm \dd V_1$ bounded by~$C$.
Then $C$ is inessential in~$\dd V_1$ because $\dd V_1$ is incompressible in~$S^3\sm K$ (Claim~\ref{cl:Vi-incompressible}).
Let $\delta$ be the open disk in~$\dd V_1$ bounded by~$C$ ($\delta$~may intersect $\dd V_2$).
Then the sphere $d\cup\delta\cup C$ bounds a ball (say,~$B$) in $S^3\sm K$. 
We have $B\cap \dd V_1 = \delta\cup C$.
It follows that we can eliminate $C$ (together with $\delta\cap \dd V_2$, if nonempty) by an isotopy of $\dd V_1$ in a neighborhood of~$B$.
Therefore, we can eliminate all components of $\dd V_1\cap \dd V_2$ that are inessential in~${\dd V_2}$.
The remaining curves of $\dd V_1\cap \dd V_2$ are essential in $\dd V_1$ as well.
(For if $C$ is an innermost of inessential curves from $\dd V_1\cap \dd V_2$ on $\dd V_1$,
then $C$ is inessential in~$\dd V_2$ because $\dd V_2$ is incompressible in~$S^3\sm K$ by Claim~\ref{cl:Vi-incompressible}.)
Now, if $\dd V_1\cap \dd V_2$ is still nonempty, the space $\dd V_1\sm \dd V_2$ is a collection of annuli.
It~is known that every incompressible properly embedded annulus in the closure of the complement of a hyperbolic knot is boundary parallel (see, e.\,g., \cite[Lemma\,15.26]{BZ06}). 
Applying this to the space $S^3\sm \inr(V_2)$,
we see that there exists an isotopy of $\dd V_1$ in $S^3\sm K$ moving $\dd V_1$ in $S^3\sm \dd V_2$.
Claim~\ref{cl:iso-dd} is proved.

\smallskip

The classical Isotopy Extension Theorem (for smooth manifolds) says that  
if $A$ is a compact submanifold of a manifold~$M$ and $F\colon A\times I\to M$ is an isotopy of~$A$ with $F(A\times I)\subset \inr(M)$,
then $F$ extends to an ambient isotopy (i.\,e., a diffeotopy of $M$) having compact support (see, e.\,g., \cite[p.\,179]{Hir76}).
Applying this theorem to the isotopy of~$\dd V_1$ in $S^3\sm K$ from Claim~\ref{cl:iso-dd} yields the following.

\begin{claim}
There exists an ambient isotopy of $S^3$, fixing $K$ pointwise, that moves $V_1$ to a position in which $\dd V_1\cap \dd V_2=\emptyset$.
\end{claim}

Thus, we can assume without loss of generality that $\dd V_1 \cap \dd V_2 = \emptyset$ (while $V_1$ and $V_2$ satisfy all properties listed at the beginning of the proof).
Now, let $M_1$ and $M_2$ denote the closures of the complements $S^3\sm V_1$ and $S^3\sm V_2$ respectively.

\begin{claim}
\label{cl:disj}
$M_1$ and $M_2$ are disjoint.
\end{claim}

In order to prove Claim~\ref{cl:disj}, we need the following assertion.

\begin{claim}
\label{cl:isotV1V2j}
There is no isotopy between $\dd V_1$ and $\dd V_2$ in $S^3\sm K$.
\end{claim}


Suppose to the contrary that such an isotopy exists.
Then the Isotopy Extension Theorem (see above) implies that
there exists an ambient isotopy of $S^3$, fixing $K$ pointwise, that moves $\dd V_1$ to $\dd V_2$.
This yields an isotopy between $V_1$ and $V_2$ that fixes $K$ pointwise.
Then the triples $(V_1,K,\ell_1)$ and $(V_2,K,\ell_2)$, where $\ell_i$ is a longitude of~$V_i$, $i=1,2$, are homeomorphic, i.\,e.,
there exists a homeomorphism $\tau\colon V_1\to V_2$ such that $\tau(K)=K$ and $\tau(\ell_1)=\ell_2$.
This implies that the pairs $(S^3,\phi_1(K))$ and $(S^3,\phi_2(K))$ are homeomorphic. 
Indeed, we observe that $(S^3,\phi_i(K))$ is obtained from $(V_i, K)$ by a \emph{Dehn filling} along $\ell_i$, that is, 
$(S^3,\phi_i(K))$ is obtained by attaching a solid torus $V$ to $V_i$ by a gluing homeomorphism $\sigma_i\colon \dd V\to \dd V_i$ such that $\sigma_i^{-1}(\ell_i)$ bounds a meridional disk of~$V$.  
Thus, the homeomorphism $\tau\colon V_1\to V_2$ extends to a homeomorphism $S^3\to S^3$ that maps $\phi_1(K)$ to $\phi_2(K)$.
This means that the knots $H_1$ and $H_2$ are equivalent because we have $\phi_i(K)=H_i$ by construction.
This contradicts the assumption that $H_1$ and $H_2$ are distinct. Claim~\ref{cl:isotV1V2j} is proved.

Now, we pass to the proof of Claim~\ref{cl:disj}.
Observe that neither $M_1$ contains $\dd M_2=\dd V_2$ nor $M_2$ contains $\dd M_1=\dd V_1$ because an incompressible torus in a hyperbolic knot complement is boundary parallel by Thurston’s hyperbolization theorem, while $\dd M_1=\dd V_1$ and $\dd M_2=\dd V_2$ are not parallel by Claim~\ref{cl:isotV1V2j}.
Obviously, this implies that $M_1$ and $M_2$ are disjoint.

Another fact that we need is implied by the following proposition.

\begin{prop}
\label{prop:Budney}
Let $C_1$, $C_2$, \dots, $C_n$ be $n$ disjoint submanifolds of $S^3$ 
such that for all $i \in \{1, 2, \dots , n\}$, $K_i =\clos(S^3 \sm C_i)$ is a non-trivially embedded solid-torus in~$S^3$. 
Then there exists n disjointly embedded 3-balls $B_1$, $B_2$, \dots, $B_n \subset S^3$ such that $C_i \subset B_i$ for all $i \in \{1, 2, \dots , n\}$. 
Moreover, each $B_i$ can be chosen to be $C_i$ union a $2$-handle which is a tubular neighbourhood of a meridional disk for $K_i$.
\end{prop}

\begin{proof}
See {\cite[Proposition\,2.1]{Bud06}} and references therein for earlier proofs.
\end{proof}

Applying Proposition~\ref{prop:Budney} to $M_1$ and $M_2$, we obtain the following claim.

\begin{claim}
There exists a meridional disk $D_2$ for $V_2$ such that $D_2 \subset V_1$.
\end{claim}

Now, since we have $M_2\subset V_1$ (Claim~\ref{cl:disj}), the image $\phi_1(M_2)$ is well defined.
We consider the complement $W:=S^3\sm \phi_1 (\inr(M_2))$.
Due to Alexander's theorem on embedded torus in~$S^3$, we observe that $W$ is a knotted solid torus because we know that the boundary $\dd W= \dd \phi_1(M_2) = \phi_1(\dd M_2)=\phi_1(\dd V_2)$ is a torus, while the complement $S^3\sm W=\phi_1 (\inr(M_2))$ is homeomorphic to $\inr(M_2)$, which is the complement of the knotted solid torus~$V_2$. (Of course, by the Gordon--Luecke theorem we know, moreover, that $W$ is a tubular neighbourhood of a hyperbolic knot.)
We see that $W$ contains $\phi_1(K)$ by construction.
Finally, we see that the winding number of $\phi_1(K)$ in~$W$ is equal to the winding number of $K$ in~$V_2$ because there exists a meridional disk $D_2$ for $V_2$ such that $D_2 \subset V_1$ so that $\phi_1$ maps $D_2$ to a meridional disk of~$W$. 
Therefore, $\phi_1(K)$ is contained in a knotted solid torus~$W$ and the winding number of $\phi_1(K)$ in~$W$ is at least~$2$.
This means that $\phi_1(K)$ is a satellite knot. 
Since we have $H_1=\phi_1(K)$, this contradicts the assumption that~$H_1$ is not a satellite knot.
This contradiction completes the proof of assertion~(ii) of Proposition~\ref{lem:main}.

\section{A combinatorial lemma}
\label{sec:combinatorial_lemma}

The present section contains a lemma which is used in the proof of assertion (iii) of Proposition~\ref{lem:main}.

\subsection*{Definitions}
Let $K$ be a knot in the $3$-sphere $S^3=\R^3\cup\{\infty\}$,  
and let $D\subset S^2$ be a projection of~$K$ on a $2$-sphere $S^2=\R^2\cup\{\infty\}$ in~$S^3$. 
A knot projection is said to be \emph{regular} if its only singularities are transversal double points.
If $D$ is a regular knot projection, an \emph{edge} in~$D$ is the closure of a component of the set $D\sm V$, where $V$ is the set of double points of~$D$.
We say that two edges $I$ and $J$ of~$D$ are \emph{neighboring edges} or \emph{neighbors}
if there exists a component $Q$ of $S^2\sm D$ such that the boundary $\dd Q$ contains both $I$ and $J$.
We say that two edges $I$ and $J$ of~$D$ are \emph{consecutive} if the union $I\cup J$ is the image of a (connected) arc of the knot. 
We will denote by $\rho$ the maximal metric on the set $E(D)$ of edges of~$D$ in the class of metrics satisfying the condition 
$$
\rho(I,J)=1\quad\text{if $I$ and $J$ are consecutive edges of~$D$.}
$$

\begin{lem}\label{lem:Berge-Coroll}
Any regular knot projection with $n>0$ double points has a pair of neighboring edges $I$ and $J$ with $\rho(I,J)\ge 2n/3$.
\end{lem}

\begin{proof}
Let $D\subset S^2$ be a regular knot projection with $n$ double points.
We consider the case with $n\ge2$ (the case $n=1$ is obvious).
Observe that $D$ has $2n$ edges. 
Put $k:=\lfloor 2n/3\rfloor$, the largest integer not greater than $2n/3$, and split the set $E(D)$ of edges of~$D$ in three parts, $E_1$, $E_2$, and $E_3$, such that each part is a chain of consecutive edges, two parts consist of $k$ edges each, and the third part consists of $2n-2k$ edges. (Note that $2n-2k\in\{k,k+1,k+2\}$; in particular, we have $2n-2k\ge k$, that is, each part consists of at least $k$ edges. No part is empty since we assume $n\ge 2$.)
Let $D_i\subset D$, $i=1,2,3$, be the union of edges from~$E_i$. 
Observe that each $D_i$ is compact and connected and $D=D_1\cup D_2\cup D_3$.
Let us smoothly embed $S^2$ in~$\R^3$ as a sphere of radius $1$ and let $\dist$ denote the metric on~$S^2$ induced by the euclidean metric in~$\R^3$.
For each $i\in\{1,2,3\}$ we set
$$
R_i:=\{x\in S^2 : \dist(x,D_i)=\dist(x,D)\}.
$$
Observe that $R_1\cup  R_2\cup R_3= S^2$ because $D=D_1\cup D_2\cup D_3$.
We see that for each $i\in\{1,2,3\}$ the set $R_i$ is closed because $D_i$ is compact (consider a convergent sequence of points in~$R_i$).
Also, we see that for each $i\in\{1,2,3\}$ the set $R_i$ is connected. 
Indeed, if $p\in R_i$, then due to compactness of $D_i$ there exists a point $q\in D_i$ such that $\dist(p,q)=\dist(p,D)$. 
Then the geodesic segment between $p$ and $q$ is in $R_i$ by the triangle inequality.
Therefore, $R_i$ is connected because $D_i$ is connected.
Finally, we see that for any $\{i,j\}\subset \{1,2,3\}$ the intersection $R_i\cap R_j$ is not empty because $D_i\subset R_i$ and $D_j\subset R_j$, while $D_i\cap D_j$ is not empty.

Thus, the sets $R_1$, $R_2$, and $R_3$ satisfy assumptions of Lemma~\ref{lem:Berge} below.
Lemma~\ref{lem:Berge} implies that $R_1$, $R_2$, and $R_3$ have a common point~$x$.
Clearly, $x$ is not an inner point of an edge of~$D$, so we have two possible cases:

1) $x$ is a double point of~$D$,

2) $x\notin D$. 

Suppose $x$ is a double point of~$D$. 
Then there exists a triple $\{J_1, J_2, J_3\}$ of edges of $D$ incident to~$x$ such that $J_i\in E_i$ for all $i\in\{1,2,3\}$. Without loss of generality we can and will assume that $J_1$ and $J_3$ are consecutive. 
Then $J_1$ and $J_2$ are neighbors, and $J_2$ and $J_3$ are neighbors.
It is easily seen that we have $\rho(J_1,J_2)\ge k$ and if $\rho(J_1,J_2)= k$ then $\rho(J_2,J_3)=k+1$, and the theorem follows.

Suppose $x\in S^2\sm D$. Let $Q$ be the component of $S^2\sm D$ containing~$x$. 
Observe that the set 
$$
\{y\in D : \dist (x,y)=\dist (x,D)\}
$$
is contained in $\dd Q \subset D$ and contains no double points of~$D$ (due to smoothness of embedding $S^2\to \R^3$). 
Therefore, since $x\in R_1\cap R_2\cap R_3$,  for each $i\in\{1,2,3\}$ the set $\dd Q\cap D_i$ contains at least one edge of~$D$. 
This means that there exists a triple $\{J_1, J_2, J_3\}$ of pairwise neighboring edges of $D$ such that we have $J_i\in E_i$ for all $i\in\{1,2,3\}$. 
It is an easy exercise to check that this triple contains a pair $\{I,J\}$ with\footnote{We use notation $\lceil 2n/3\rceil$ for the smallest integer not less than~$2n/3$.} 
$$\rho(I,J)\ge \lceil 2n/3\rceil\ge 2n/3.\qedhere$$
\end{proof}

\begin{lem}\label{lem:Berge}
If a triple of pairwise intersecting closed connected sets cover a simply connected space, then these three sets have a common point.
\end{lem}

\begin{proof}
This follows, e.\,g., from Theorem~5 of \cite{Bog02} in the case $m=1$.
\end{proof}

\section{Tangles}
\label{sec:tangles}

Our proof of assertion (iii) of Proposition~\ref{lem:main} uses tangles.
The present section contains some preliminaries on tangles.

\subsection*{Definitions} 
A $k$-string \emph{tangle}, where $k\in\mathbb N$, is a pair $(B, t)$ where $B$ is a $3$-ball and $t$ is the union of $k$ disjoint arcs in $B$ with $t \cap \dd B = \dd t$. We mostly interested in the cases where $k\in\{1,2\}$.
Two tangles, $(B, t)$ and $(A, s)$, are \emph{equivalent} if there is a homeomorphism of pairs from $(B, t)$ to $(A, s)$. 
A~tangle $(B, t)$ is \emph{trivial} if $B$ contains a properly embedded disk containing~$t$.
A~tangle $(B, t)$ is \emph{locally knotted} 
if $B$ contains a ball $B'$ such that $(B',B'\cap t)$ is a nontrivial $1$-string tangle.
A~$2$-string tangle $(B, t)$ is \emph{prime} if it is neither locally knotted nor trivial.
If $(B, t)$ and $(A, s)$ are $k$-string tangles and $f\colon (\dd B, \dd t) \to (\dd A, \dd s)$ is a homeomorphism, 
a link in $S^3$ can be obtained by identifying the boundaries of the tangles using~$f$. 
The result, $(B, t)\cup_f (A, s)$, is referred to as a \emph{sum} of the two tangles. 
If $(B,t)$ is a $1$-string tangle and $(A,s)$ is the trivial $1$-string tangle, then there is a unique (up to a homeomorphism of pairs) knot which is a sum of $(B,t)$ and $(A,s)$. This knot is called the \emph{closure} of $(B,t)$. 
We say that a $2$-string tangle $(B, t)$ is a \emph{cable} tangle if 
there exists an embedding $f\colon I \times I \to B$ such that $f(I \times I)\cap \dd B= I\times \dd I$ and $t=f(\dd I\times I)$, where $I:=[0,1]$.
(We treat the trivial $2$-string tangle as a cable tangle.)
Clearly, each tangle $(B, t)$ can be embedded in $\mathbb{R}^3$ in such a way that $B$ becomes a Euclidean ball while the endpoints $\dd t$ lie on a great circle of this ball
and $t$ is in general position with respect to the projection onto the flat disc bounded by the great circle. 
The projection, with additional information of over- and undercrossings, then gives us a \emph{tangle diagram}. Examples of tangle diagrams are given in Figs.~\ref{fig:gamma}, \ref{fig:gamma-tref}, and~\ref{fig:Reid+gamma}.

\begin{thm}[{\cite[Theorem\,1]{Lick81}}]\label{thm:sum-prime}
A sum of two $2$-string prime tangles is a prime link.
\end{thm}

\begin{lem}\label{lem:cable-prime}
Each nontrivial cable $2$-string tangle is prime.
\end{lem}

\begin{proof} (See~\cite[Examples~(a) and~(b)]{Lick81}.)
It is enough to observe that we can, in an obvious manner, add the trivial $2$-string tangle to any cable $2$-string tangle so as to create the trivial knot, which proves that the initial tangle has no local knots (this follows by the Unique Factorization Theorem by Schubert~\cite{Schu49}).
\end{proof}

\begin{lem}\label{lem:cable-triv-prime}
No composite knot is a sum of a nontrivial cable $2$-string tangle with the trivial $2$-string tangle.
\end{lem}

\begin{proof}
Suppose that a knot $K$ in $S^3$ is presented as a sum 
$$
(S^3,K)=(B, t)\cup_f (A, s), \quad f\colon (\dd A, \dd s) \to (\dd B, \dd t),
$$ 
of a nontrivial cable $2$-string tangle $(B,t)$ with a trivial $2$-string tangle $(A, s)$.
Let $$f_0\colon (\dd A, \dd s) \to (\dd B, \dd t)$$ yields an obvious `trivializing' sum for $(B,t)$, that is, the sum $(B, t)\cup_{f_0} (A, s)$ is the trivial knot.\footnote{In fact, the results of \cite{BS86, BS88} imply that there is essentially unique way to create the trivial knot as a sum of a given prime $2$-string tangle and a trivial $2$-string tangle. In particular, if $\phi\colon (\dd A, \dd s) \to (\dd B, \dd t)$ is a homeomorphism such that $(B, t)\cup_\phi (A, s)$ is the trivial knot then the map $f_0^{-1}\circ\phi\colon (\dd A, \dd s)\to (\dd A, \dd s)$ extends to a map $F\colon(A,s)\to(A,s)$ such that $F(s)=s$. 
}
(See left side of Fig.~\ref{fig:333}.)
Let $M_0$ denote the double cover of the $3$-sphere $B\cup_{f_0} A$ branched over the trivial knot $t\cup_{f_0} s$,
and let $M_1$ be the double cover of the $3$-sphere $B\cup_f A$ branched over the knot 
$t\cup_f s=K$.
Then $M_0$ is homeomorphic to the $3$-sphere, while
$M_0$ and $M_1$ are related by a Dehn surgery along the solid torus covering~$(A,s)$.
We observe that the solid torus $V_A\subset M_0=S^3$ that covers~$(A,s)$ is knotted as a composite knot. 
Indeed, the definition of cable tangles imply that there is an obvious ambient isotopy of $B\cup_{f_0} A$ that moves $t\cup_{f_0} s$ and $A$ to a position in which $t\cup_{f_0} s$ is a geometric circle and $A$ is a closed regular neigborhood of a `knotted diameter' of this circle. See Fig.~\ref{fig:333}.

\begin{figure}[ht]
\begin{center}
\begin{tikzpicture} 
[bballs/.style={white,double=blue!20!white,ultra thick,double distance=5pt},
knots/.style={white,double=black,ultra thick,double distance=1pt}]

\fill[blue!20!white](0,0) circle (2/3) node [color=black]{$A$};
\draw[knots] 
(0/15,20/15) .. controls +(7/15,0/15) and +(0/15,-6/15) .. (15/15,31/15)
(0/15,25/15) .. controls +(5/15,0/15) and +(0/15,-3/15) .. (10/15,31/15);
\draw[knots] 
(11/15,15/15) .. controls +(0/15,20/15) and +(0/15,12/15) .. (-15/15,33/15)
(6/15,15/15) .. controls +(0/15,18/15) and +(0/15,10/15) .. (-11/15,32/15);
\draw[knots] 
(15/15,31/15) .. controls +(0/15,17/15) and +(0/15,25/15) .. (-12/15,19/15)
(10/15,31/15) .. controls +(0/15,10/15) and +(0/15,12/15) .. (-8/15,20/15);
\draw[knots] 
(0/15,20/15) .. controls +(-7/15,0/15) and +(0/15,-8/15) .. (-15/15,33/15)
(0/15,25/15) .. controls +(-5/15,0/15) and +(0/15,-5/15) .. (-11/15,32/15);
\draw[line width=1pt] 
(-12/15,19/15) .. controls +(0/15,-10/15) and +(0/15,-5/15) .. (-3/15,8/15)
(-8/15,20/15) .. controls +(0/15,-7/15) and +(0/15,2/15) .. (-3/15,8/15);
\draw[line width=1pt] 
(11/15,15/15) .. controls +(0/15,-10/15) and +(0/15,-4/15) .. (5/15,6/15)
(6/15,15/15) .. controls +(0/15,-5/15) and +(0/15,2/15) .. (5/15,6/15);

\draw[->] (1.6,2.7) arc (120:60:.8);

\begin{scope}
[xshift=4cm] 

\draw[->] (1.6,2.7) arc (120:60:.8);

\fill[blue!20!white](0,0) circle (2/3);
\draw[bballs] 
(0/15,20/15) .. controls +(7/15,0/15) and +(0/15,-6/15) .. (15/15,31/15);
\draw[bballs] 
(11/15,19/15) .. controls +(0/15,20/15) and +(0/15,12/15) .. (-15/15,33/15)
(6/15,19/15) .. controls +(0/15,18/15) and +(0/15,10/15) .. (-15/15,33/15);
\draw[line width=5pt, color=blue!20!white] 
(8/15,19/15) .. controls +(0/15,20/15) and +(0/15,11/15) .. (-15/15,33/15)
(9/15,19/15) .. controls +(0/15,20/15) and +(0/15,11/15) .. (-15/15,33/15);
\draw[bballs] 
(15/15,31/15) .. controls +(0/15,17/15) and +(0/15,25/15) .. (-12/15,19/15)
(15/15,31/15) .. controls +(0/15,17/15) and +(0/15,25/15) .. (-10/15,20/15);
\draw[bballs] 
(0/15,20/15) .. controls +(-7/15,0/15) and +(0/15,-8/15) .. (-15/15,33/15);
\draw[line width=5pt, color=blue!20!white]
(-12/15,19/15) .. controls +(0/15,-7/15) and +(0/15,7/15) .. (-8.7/15,0/15)
(-10/15,20/15) .. controls +(0/15,-7/15) and +(-5/15,0/15) .. (0/15,8.7/15)
(-10/15,20/15) .. controls +(0/15,-7/15) and +(-5/15,0/15) .. (-2/15,7/15)
(-10/15,20/15) .. controls +(0/15,-7/15) and +(-5/15,0/15) .. (-4/15,5/15);
\fill[blue!20!white](8.47/15,19/15) circle (3.78/15);
\fill[blue!20!white](0/15,20/15) circle (2.4pt);
\fill[blue!20!white](-15/15,33/15) circle (2.4pt);
\fill[blue!20!white](15/15,31/15) circle (2.4pt);
\fill[blue!20!white](-12/15,19/15) circle (2.4pt);
\fill[blue!20!white](-10/15,19.2/15) circle (2.4pt);
\draw[line width=1pt] 
(11/15,15/15) .. controls +(0/15,-10/15) and +(0/15,-4/15) .. (5/15,6/15)
(6/15,15/15) .. controls +(0/15,-5/15) and +(0/15,2/15) .. (5/15,6/15);
\draw[line width=1pt] 
(11/15,15/15) .. controls +(0/15,2/15) and +(2/15,0/15) .. (8.5/15,19/15)
(6/15,15/15) .. controls +(0/15,2/15) and +(-2/15,0/15) .. (8.5/15,19/15);
\draw (-7/15,7/15) node{$A$};

\end{scope}

\begin{scope}
[xshift=8.7cm,yshift=-0.2cm] 

\draw[bballs] 
(0/22,20/22) .. controls +(15/22,0/22) and +(15/22,0/22) .. (5/22,43/22);
\draw[bballs] 
(-8/22,41/22) .. controls +(10/22,0/22) and +(-35/22,0/22) .. (38/22,30/22)
(-8/22,41/22) .. controls +(10/22,0/22) and +(-35/22,0/22) .. (38/22,28/22);
\draw[bballs] 
(-38/22,30/22) .. controls +(38/22,0/22) and +(-10/22,0/22) .. (5/22,43/22)
(-38/22,28/22) .. controls +(38/22,0/22) and +(-10/22,0/22) .. (5/22,43/22);
\draw[bballs] 
(0/22,20/22) .. controls +(-15/22,0/22) and +(-15/22,0/22) .. (-8/22,41/22);
\fill[blue!20!white](-38/22,29/22) circle (3.8pt);
\fill[blue!20!white](38/22,29/22) circle (3.8pt);
\fill[blue!20!white](0/22,20/22) circle (2.4pt);
\fill[blue!20!white](-8/22,41/22) circle (2.4pt);
\fill[blue!20!white](5/22,43/22) circle (2.4pt);
\draw (-34/22,29/22) node{$A$};
\draw[line width=1pt] 
(0/22,29/22) circle (38/22);

\end{scope}

\end{tikzpicture} 
\end{center}
\caption{For the proof of Lemma~\ref{lem:cable-triv-prime}} \label{fig:333}
\end{figure}
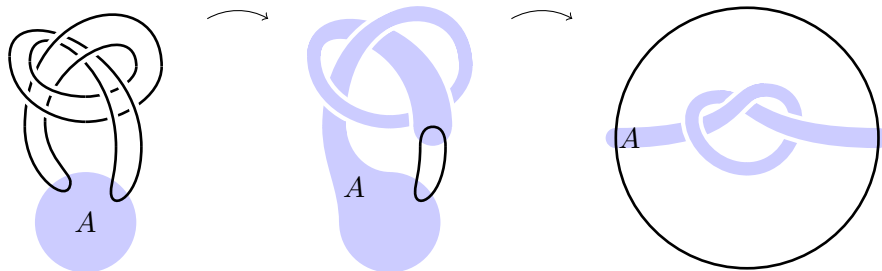

This clearly implies that $V_A$ is a regular neigborhood of a composite knot. (This composite knot is a sum of two copies of the $1$-string tangle $(B,t_1)$, where $t_1$ is a component of~$t$.) 
It is known that a nontrivial Dehn surgery on a composite knot in~$S^3$ yields an irreducible (hence prime) manifold (see~\cite[Theorem~7.1]{Gor83}). 
It is known that if the double cover of $S^3$ branched over a knot $R$ is prime then $R$ is prime (see \cite{Wal69}; see also \cite[Corollary~4]{KT80} for the inverse implication).
Consequently, $K$ is a prime knot if nontrivial.
\end{proof}

\subsection*{Remarks} 
1. Lemma~\ref{lem:cable-triv-prime} also follows from results of~\cite{E-M86} (see also~\cite[Theorem~6]{E-M88})
or equivalently from the fact that \emph{only integral Dehn surgeries can yield reducible manifolds} \cite{GL87}.
This way of proof uses the fact that cable knots are prime (see \cite[p. 250, Satz 4]{Schu53}, \cite[Cor.~2]{Gra91}).

2. Lemma~\ref{lem:cable-triv-prime} is used in the proof of Proposition~\ref{as:iii-for-PT} (which in its turn is used in the proof of assertion~(iii) of Proposition~\ref{lem:main}), where it covers the case of $2$-bridge knots.
It is known (see~\cite{Wel92}) that the percentage of $2$-bridge knots amongst all of the 
prime knots of~$n$ or fewer crossings approaches $0$ as $n$ approaches infinity.
Thus, in the proof of Theorem~\ref{T1}, we can discard $2$-bridge knots together with Lemma~\ref{lem:cable-triv-prime}. Nevertheless, we use Lemma~\ref{lem:cable-triv-prime} for the sake of completeness of Propositions~\ref{as:iii-for-PT} and~\ref{lem:main}.

\begin{cor}\label{cor:cable-triv-prime}
No composite knot is a sum of two cable $2$-string tangles.
\end{cor}

\begin{proof}
A cable $2$-string tangle is either prime or trivial (Lemma~\ref{lem:cable-prime}).
A sum of two prime $2$-string tangles is a prime link by Theorem~\ref{thm:sum-prime}.
No composite knot is a sum of a nontrivial cable $2$-string tangle with the trivial $2$-string tangle by Lemma~\ref{lem:cable-triv-prime}.
If a knot $K$ is a sum of two trivial $2$-string tangles,
then the bridge number $\bn(K)$ of~$K$ is at most~$2$.
If~$\bn(K)=1$ then $K$ is a trivial knot.
If~$\bn(K)=2$ then $K$ a prime knot by~\cite{Schu54}.
\end{proof}

\section{Proof of assertion $\rm(iii)$ of Proposition~\ref{lem:main}}
\label{sec:proof-iii}

\subsection*{Definition. Weak property PT}
Let $D$ be a knot diagram on the $2$-sphere $S^2=\R^2\cup\{\infty\}$.
We say that $D$ has \emph{weak property PT} (PT stands for `tangle primeness') if 
$D$ is obtained by adding ears to a diagram of a tangle that is not locally knotted (that is, the tangle is either prime or trivial).
In other words, $D$ has weak property PT if there exists a $2$-disk $d\subset S^2$ such that 

-- the boundary $\dd d$ intersects $D$ transversely in four points;

-- the intersection $d\cap D$ consists of two simple non-intersecting arcs (as on the left side of Fig.~\ref{fig:gamma}); 

-- the complementary disk $\delta:=S^2\sm \inr(d)$ with the diagram $\delta \cap D$ represents a $2$-string tangle $(B,t)$ which is not locally knotted (that is, $(B,t)$ is either prime or trivial).

We say that a knot has \emph{weak property PT} if it has a minimal diagram with weak property PT.

\begin{prop}\label{prop:23impliesPT}
Each minimal diagram of each $\frac23$-regular prime knot has weak property~PT.
In particular, each $\frac23$-regular prime knot has weak property~PT.
\end{prop}

\begin{proof}
Let $D_P$ be a minimal diagram, on the $2$-sphere $S^2=\R^2\cup\{\infty\}$, of a $\frac23$-regular prime knot~$P$.
By Lemma~\ref{lem:Berge-Coroll}, $D_P$ has a pair of neighboring edges $I$ and $J$ with $\rho(I,J)\ge  \frac{2\crn(P)}3$.
Since $I$ and $J$ are neighbors, there exists a disk $d\subset S^2$ such that the intersection $d \cap D_P$ consists of a subarc of $I$ and a subarc of~$J$, while $\dd d$ intersects $D_P$ transversely in four points. 
Let $\delta$ denote the disk $S^2\sm \inr(d)$, and let
$(B,t)$ be the $2$-string tangle represented by the diagram $\delta \cap D_P$.
Let $t_1$ and $t_2$ be the components of~$t$, and let $K_1$ and $K_2$ be the knots that are the closures of the $1$-string tangles $(B,t_1)$ and $(B,t_2)$.

\begin{claim}
\label{clm:2/3}
We have $\crn(K_i)\le \frac23\crn(P)-1$ for $i\in\{1,2\}$.
\end{claim}

\begin{proof}
The diagram $\delta \cap D_P$ of the tangle $(B,t)$ is formed by two curves, $c_1$ and $c_2$ say,
corresponding to the components $t_1$ and $t_2$, respectively, of~$t$.
We denote by $\crn(c_i)$ the number of double points of~$c_i$.
Since a diagram of~$K_1$ can be obtained from $c_1$ by adding a simple arc in~$d$, it follows that we have
\begin{equation} 
\label{eq:Ma1}
\crn(K_1)~\le~ \crn(c_1).
\end{equation}
Observe that by construction we have 
\begin{equation}
\label{eq:Paa}
\crn(P)~=~\crn(D_P)~=~\crn(c_1)+\crn(c_2)+\card(c_1\cap c_2).
\end{equation}
By the definition of~$\rho$ (this definition is given at the beginning of Sec.~\ref{sec:combinatorial_lemma}) we have
\begin{equation}
\label{eq:rhomin}
\rho(I,J)=\min\{2\crn(c_1)+\card(c_1\cap c_2), 2\crn(c_2)+\card(c_1\cap c_2)\}.
\end{equation}
Since $\rho(I,J)\ge  \frac{2\crn(P)}3$, it follows from \eqref{eq:Ma1}, \eqref{eq:Paa},  and~\eqref{eq:rhomin} that
\begin{equation*}
\crn(K_1)~\le~ \crn(c_1) ~\le~  \frac23\crn(P)-\frac{\card(c_1\cap c_2)}2.
\end{equation*}
Since $D_P$ is a minimal diagram of a prime knot and $I\neq J$, it follows that ${c_1\cap c_2\neq \emptyset}$.
Assuming that $c_1$ intersects $c_2$ in a unique point ($q$, say) implies that $q$ is a 
cutpoint\footnote{A point $x$ of a connected topological space~$X$ is a  \emph{cutpoint} if the set $X\sm\{x\}$ is not connected.} of~$D_P$. However, no minimal diagram of a knot has a cutpoint.
This implies that $\card(c_1\cap c_2)\ge 2$ and $\crn(K_1)\le \frac23\crn(P)-1$, as required.
The case of $K_2$ is analogous.
\end{proof}

\begin{claim}
\label{clm:prime-or-trivial}
The $2$-string tangle $(B,t)$ represented by the diagram $\delta \cap D_P$ is either prime or trivial.
\end{claim}

\begin{proof}
Suppose on the contrary that $(B,t)$ is neither prime nor trivial.
Then $(B,t)$ is locally knotted, that is,
$B$ contains a ball $A$ such that the pair $(A,A\cap t)$ is a nontrivial $1$-string tangle.
Let $t_i$, where $i\in\{1,2\}$, be the component of~$t$ that meets~$A$.
We denote by~$L$ the knot that is the closure of the $1$-string tangle ${(A,A\cap t)}$.
Then $L$ is a factor of~$P$.
Since $P$ is prime and $L$ is nontrivial, it follows that $L$ and $P$ are equivalent.
At the same time, $L$ is a factor of~$K_i$ (as defined above, $K_i$ is the closure of the $1$-string tangle $(B,t_i)$).
Since $L$ and $P$ are equivalent, while $P$ is assumed to be $\frac23$-regular,
we have $\crn(K_i)\ge\frac23\crn(P)$, which contradicts Claim~\ref{clm:2/3}.
The~obtained contradiction proves that $(B,t)$ is either prime or trivial.
\end{proof}

Thus, all requirements from the definition of weak property~PT are fulfilled.
Consequently, $D_P$ has weak property~PT.
Proposition~\ref{prop:23impliesPT} is proved.
\end{proof}

\begin{prop}
\label{as:iii-for-PT}
If $P$ is a knot with weak property PT, then there exists a \underline{prime} $\gamma$-knot $P'$ over~$P$ with $\crn(P')\le \crn(P)+17$.
\end{prop}

\begin{proof}
By definition, $P$ has a minimal crossing diagram $D_P$ with weak property~PT.
This means that there exists a disk $d\subset S^2$ such that

-- the boundary $\dd d$ intersects $D_P$ transversely in four points;

-- the intersection $d\cap D_P$ consists of two simple non-intersecting arcs;

-- the tangle diagram $\delta \cap D_P$, where $\delta:=S^2\sm \inr(d)$, represents 
either prime or trivial $2$-string tangle $(B,t)$.

Without loss of generality we can identify the pair $(d, d \cap D_P)$ with the tangle diagram in Fig.~\ref{fig:gamma}(a).
We have the following two cases:

($\alpha$) two arrows on the arcs in Fig.~\ref{fig:gamma}(a) indicate the same orientation on~$P$,

($\beta$) two arrows on the arcs in Fig.~\ref{fig:gamma}(a) induce opposite orientations on~$P$.

In case ($\alpha$), let $D_\alpha$ be the diagram obtained from $D_P$ by local move as in Fig.~\ref{fig:gamma}
and let $P_\alpha$ be the knot represented by~$D_\alpha$.
Since the figure-eight knot is hyperbolic, an easy argument shows that $P_\alpha$ is a $\gamma$-knot over~$P$.
We check that $P_\alpha$ has all of the desired properties.
First, the obtained diagram~$D_\alpha$ of~$P_\alpha$ has $\crn(P)+16$ crossings. This means that $\crn(P_\alpha)\le \crn(P)+16$.
Next, we prove that $P_\alpha$ is prime.
We observe that, by construction, $P_\alpha$ is a sum of the cable tangle of Fig.~\ref{fig:gamma}(b) and the tangle $(B,t)$, which is prime or trivial.
Each nontrivial cable tangle is prime (see Lemma~\ref{lem:cable-prime}).
If $(B,t)$ is prime then $P_\alpha$ is prime by Theorem~\ref{thm:sum-prime}.
If $(B,t)$ is trivial then $P_\alpha$ is prime by Lemma~\ref{lem:cable-triv-prime} and assertion~(i) of Proposition~\ref{lem:main}
(Lemma~\ref{lem:cable-triv-prime} implies that $P_\alpha$ is either prime or trivial if $(B,t)$ is trivial; 
assertion~(i) implies that $P_\alpha$ is a satellite knot and hence nontrivial).
Thus, $P_\alpha$ is a prime $\gamma$-knot over $P$ with $\crn(P_\alpha)\le \crn(P)+16$, as required.

In case ($\beta$), let $D_\beta$ be the diagram obtained from $D_P$ by local move as in Fig.~\ref{fig:Reid+gamma}
and let $P_\beta$ be the knot represented by~$D_\beta$. 

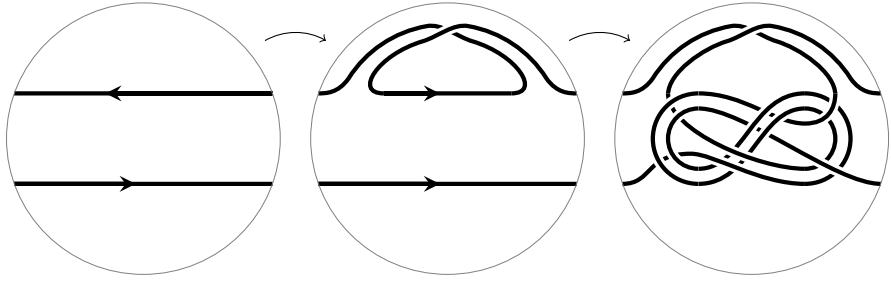
\begin{figure}[ht]
\begin{center}

\begin{tikzpicture} 
[knots/.style={white,double=black,very thick,double distance=1.6pt}]
\draw[line width=1.6pt]  [-stealth]
(-1/5,3/5) -- (16/5,3/5) 
(16/5,3/5)  -- (5/5,3/5);
\draw[line width=1.6pt] [-stealth]
(-1/5,-3/5) -- (16/5,-3/5)
(-1/5,-3/5) -- (7/5,-3/5);
 \draw[help lines](7.5/5,0) circle (9/5);

\draw[->] (7.5/5+2-0.4,1.3) arc (120:60:.8);

\begin{scope}
[xshift=4cm] 

%
%
\draw[knots] 
(-1/5,3/5) .. controls +(.5/5,0) and +(-.7/5,-1/5) .. (1/5,4/5) 
arc (148:100:7.6/5)
 .. controls +(1/5,.2/5) and +(-1/5,.3/5) .. ++(3/5,-1/5)
  .. controls +(3/5,-.9/5) and +(2/5,0) .. (11.7/5,3/5);
\draw[knots] 
(16/5,3/5) .. controls +(-.5/5,0) and +(.7/5,-1/5) .. (14/5,4/5)
arc (32:80:7.6/5) 
.. controls +(-1/5,.2/5) and +(1/5,.3/5) .. ++(-3/5,-1/5)  
.. controls +(-3/5,-.9/5) and +(-2/5,0) .. (3.3/5,3/5);

\draw[line width=1.6pt] [-stealth]
(3.3/5,3/5) -- (11.7/5,3/5)
(3.3/5,3/5) -- (7/5,3/5);

\draw[line width=1.6pt] [-stealth]
(-1/5,-3/5) -- (16/5,-3/5)
(-1/5,-3/5) -- (7/5,-3/5);

\draw[help lines](7.5/5,0) circle (9/5);

\draw[->] (7.5/5+2-0.4,1.3) arc (120:60:.8);

\end{scope}

\begin{scope}
[xshift=8cm] 
%
%
%
\draw[knots] 
(-1/5,3/5) .. controls +(.5/5,0) and +(-.7/5,-1/5) .. (1/5,4/5) 
arc (148:100:7.6/5)
 .. controls +(1/5,.2/5) and +(-1/5,.3/5) .. ++(3/5,-1/5)
 .. controls +(1.5/5,-.45/5) and +(0,1.5/5) .. (13/5,2.9/5);
\draw[knots] 
(16/5,3/5) .. controls +(-.5/5,0) and +(.7/5,-1/5) .. (14/5,4/5)
arc (32:80:7.6/5) 
.. controls +(-1/5,.2/5) and +(1/5,.3/5) .. ++(-3/5,-1/5)  
.. controls +(-1.5/5,-.45/5) and +(0,1.5/5) .. (2/5,2.9/5);
%
%
\draw[knots]
(11/5,-2/5) 
.. controls +(1.111/5,0) and +(0,-1.111/5) .. ++(2/5,2/5)
.. controls +(0,1.111/5) and +(1.111/5,0) .. ++(-2/5,2/5)
(11/5,-3/5) 
.. controls +(1.666/5,0) and +(0,-1.666/5) .. ++(3/5,3/5)
 .. controls +(0,1.666/5) and +(1.666/5,0) .. ++(-3/5,3/5);
%
%
\draw[knots] 
(4/5,3/5)  .. controls +(3/5,0) and +(-2/5,0) .. 
(11/5,1/5) .. controls +(1.5/5,0) and +(0,-1/5) .. 
(13/5,3/5) 
(4/5,2/5) .. controls +(3/5,0) and  +(-3/5,-.0/5) .. (16/5,-3/5);
%
%
\draw[knots] 
(11/5,2/5) .. controls +(-3/5,0) and +(3.5/5,0) .. (4/5,-3/5)
(11/5,3/5) .. controls +(-3.5/5,0) and +(3/5,0) .. (4/5,-2/5);
%
%
\draw[knots] 
(2/5,3/5) .. controls +(0,-2.5/5) and  +(-3/5,0) ..  (11/5,-2/5)
(-1/5,-3/5) .. controls +(2/5,0) and +(-3/5,0) .. (3.5/5,-1/5)  .. controls +(1.5/5,0) and +(-3/5,0) .. (11/5,-3/5);
%
\draw[knots] 
 (4/5,-3/5)
 .. controls +(-1.666/5,0) and +(0,-1.666/5) .. ++(-3/5,3/5)
 .. controls +(0,1.666/5) and +(-1.666/5,0) .. ++(3/5,3/5)
(4/5,-2/5)
 .. controls +(-1.111/5,0) and +(0,-1.111/5) .. ++(-2/5,2/5)
 .. controls +(0,1.111/5) and +(-1.111/5,0) .. ++(2/5,2/5);
 \draw[help lines](7.5/5,0) circle (9/5);
 \end{scope}
\end{tikzpicture} 
\end{center}
\caption{Type I Reidemeister move plus double figure-eight move} \label{fig:Reid+gamma}
\end{figure}

The local move in Fig.~\ref{fig:Reid+gamma} is the composition of a type I Reidemeister move and the move shown in Fig.~\ref{fig:gamma}.
This implies that $P_\beta$ is a $\gamma$-knot over~$P$.
Obviously, $D_\beta$ has $\crn(P)+1+16$ crossings. 
This means that $\crn(P_\beta)\le \crn(P)+17$.
The primeness of~$P_\beta$ follows by the same argument as in case~($\alpha$) because $P_\beta$ is a sum 
of a nontrivial cable tangle and the tangle $(B,t)$.
Thus, $P_\beta$ is a prime $\gamma$-knot over $P$ with $\crn(P_\beta)\le \crn(P)+17$, as required.
\end{proof}

Assertion~(iii) of Proposition~\ref{lem:main} readily follows from Proposition~\ref{as:iii-for-PT} by Proposition~\ref{prop:23impliesPT}.

\section{Addendum I: Strong property PT}
\label{sec:add-I}

In addition to weak property PT defined in Sec.~\ref{sec:proof-iii},
we introduce strong property~PT.

\subsection*{Definition. Strong property PT}
Let $D$ be a knot diagram on the $2$-sphere $S^2=\R^2\cup\{\infty\}$.
We say that a tangle $(B,t)$ is \emph{represented by a connected subdiagram of~$D$}
if there exists a $2$-disk $\delta\subset S^2$ such that the intersection $\delta\cap D$ is connected and the pair $(\delta, \delta\cap D)$, with information of under- and overcrossings inherited from~$D$, is a diagram of $(B,t)$.
We say that $D$ has \emph{strong property~PT} if every $2$-string tangle represented by a connected subdiagram of~$D$ is either prime or trivial.
We say that a knot has \emph{strong property~PT} if all of its minimal diagrams have strong property PT.

\begin{prop}\label{prop:strongPT}
1. Each minimal diagram of each $1$-regular prime knot has strong property~PT.
In~particular, each $1$-regular prime knot has strong property~PT.

2. Each minimal diagram with strong property~PT has weak property~PT.
In~particular, each knot with strong property~PT has weak property~PT.
\end{prop}

\begin{proof}
1. Assume to the contrary that a non-prime non-trivial $2$-string tangle $(B,t)$
is represented by a connected subdiagram $\delta\cap D_P$ in a minimal diagram $D_P$ of a $1$-regular prime knot~$P$.
This implies in particular that $(B,t)$ is locally knotted, that is, 
$B$ contains a ball $B'$ such that $(B',B'\cap t)$ is a nontrivial $1$-string tangle.
Let $K_1$ denote the knot obtained by the closure of $(B',B'\cap t)$.
Then $K_1$ is a factor of~$P$, which is a prime knot, so that we have $K_1=P$.
(This follows by the Unique Factorization Theorem by Schubert~\cite{Schu49}.)
On the other hand, the knot~$K_1=P$ is a factor of the knot $K_2$ obtained as the closure of~$(B,t_1)$, where $t_1$ is the component of $t$ that meets~$B'$. 
Observe that we have $\crn(K_2)\le \crn(P) - 1$ because, 
since the diagram $\delta\cap D_P$ representing $(B,t)$ is connected, 
the projection of $t_1$ has at least one crossing with the projection of the second component of~$t$.
The inequality $\crn(K_2)\le \crn(P) - 1$ implies that $K_2\neq P$.
Therefore, $K_2$ is a composite knot, $P$ is a factor of~$K_2$, and $\crn(K_2)\le \crn(P) - 1$.
This contradicts the assumption that $P$ is a $1$-regular knot.

2. Let $D$ be a minimal diagram with strong property~PT. 
If $D$ is a circle with no double points then $D$ has weak property~PT (obvious).
Assume that $D$ has double points. 
We take a double point $x$ of~$D$ and consider a disk $d\subset S^2$ in a small neighborhood of~$x$ such that the intersection $d\cap D$ consists of two non-intersecting arcs (as on the left side of Fig.~\ref{fig:gamma}).
Since $D$ is a minimal diagram, 
$x$ is not a cutpoint of~$D$.
This easily implies that the intersection $\delta \cap D$, where  $\delta:=S^2\sm \inr(d)$,
is connected.
Since $D$ has strong property~PT, it follows that the $2$-string tangle represented by the connected subdiagram $\delta \cap D$ is either prime or trivial.
This means that $D$ has weak property~PT.
\end{proof}

Propositions~\ref{prop:strongPT} and~\ref{prop:23impliesPT}
give the following dependence for properties of prime knots.
 
\begin{center}
\begin{tikzpicture}
  \matrix[row sep=0mm,column sep=2mm] {
     \node {$1$-regularity}; & \node {$\Longrightarrow$}; & \node {$\frac23$-regularity}; \\
    \node [rotate=-90]{$\Rightarrow$}; &  & \node [rotate=-90]{$\Rightarrow$}; \\
      \node {strong property PT}; & \node {$\Longrightarrow$}; & \node {weak property PT}; \\
  };
\end{tikzpicture}
\end{center}

This implications can be treated in terms of conjectures.
We consider the following conjectures.

\begin{conj}\label{con:SPT}
Each prime knot has strong property PT.
\end{conj}

\begin{conj}\label{con:WPT}
Each prime knot has weak property PT.
\end{conj}


\begin{conj}\label{conj-weak-PT}
There exist $\varepsilon>0$ and $N>0$ such that, for all $n>N$, the percentage of knots with weak property PT amongst all of the 
hyperbolic knots of~$n$ or fewer crossings is at least~$\varepsilon$.
\end{conj}

We have the following implications.

\begin{center}
\begin{tikzpicture} 
\draw (0,0) node[right]{Conj.~4 $~\Longrightarrow~$ Conj.~5 $~\Longrightarrow~$ Conj.~6};
\draw (120:1.2cm) node [left]{Conj.~2};
\draw (120:0.6cm) node [rotate=-45]{$\Longrightarrow$};
\draw (-120:1.2cm) node [left]{Conj.~3};
\draw (-120:0.6cm) node [rotate=45]{$\Longrightarrow$};
\draw ++(-120:1.2cm) +(2.3cm,0) node[right]{Conj.~7 $~\Longrightarrow~$ Conj.~8 $~\Longrightarrow~$ Conj.~9};
\draw ++(-120:0.6cm) 
++(1.8cm,0) node [rotate=-45]{$\Longrightarrow$} 
++(2.2cm,0) node [rotate=-45]{$\Longrightarrow$}
++(2.2cm,0) node [rotate=-45]{$\Longrightarrow$};
\end{tikzpicture} 
\end{center}

The implication 
Conj.~\ref{con:com} $\Rightarrow$ Conj.~\ref{con:SPT} 
follows from assertion~1 of Proposition~\ref{prop:strongPT}.
The implication 
Conj.~\ref{con:SPT} $\Rightarrow$ Conj.~\ref{con:WPT} 
follows from assertion~2 of Proposition~\ref{prop:strongPT}.
The implications Conj.~\ref{con:2/3}  $\Rightarrow$ Conj.~\ref{con:WPT} 
and 
Conj.~\ref{conj-weak}  $\Rightarrow$ Conj.~\ref{conj-weak-PT} 
follow from Proposition~\ref{prop:23impliesPT}.
The implication Conj.~\ref{con:WPT} $\Rightarrow$ Conj.~\ref{conj-weak-PT}  is obvious.

\medskip

Theorem~\ref{T1} can be strengthened in the following way.

\begin{thm}\label{T2}
Conjecture~\ref{H1} contradicts (each of) Conjectures
\ref{con:add}--\ref{conj-weak-PT}.
\end{thm}

\begin{proof}
Since each of Conjectures 
\ref{con:add}--\ref{con:WPT} 
implies Conjecture~\ref{conj-weak-PT} 
(see the system of implications before Theorem~\ref{T2}), 
it suffices to show that Conjecture~\ref{H1} contradicts Conjecture~\ref{conj-weak-PT}.
In order to prove this, we repeat verbatim 
the reduction of Theorem~{\rm\ref{T1}} to Proposition~{\rm\ref{lem:main}}
up to replacing $\frac23$-regularity with weak property PT and assertion~(iii) of Proposition~\ref{lem:main} with Proposition~\ref{as:iii-for-PT}.
\end{proof}

\section{Addendum II: Non-$\frac14$-regular knots}
\label{sec:add-II}

The main theorem of the present paper states that Conjecture~\ref{H1} 
concerning predominance of hyperbolic knots
contradicts 
the conjecture on additivity of the crossing number (of knots under connected sum)
as well as several 
weaker conjectures.
In this section, we show that Conjecture~\ref{H1}  
also contradicts an assumption that 
the conjecture on additivity has 
many strong counterexamples.

We say that a knot $P$ is \emph{non-$\lambda$-regular}, $\lambda\in\mathbb{R}$,
if there exists a knot $K$ such that $P$ is a factor of $K$ while $\crn(K)< \lambda\cdot\crn(P)$.
In this section, we prove the following theorem.


\begin{thm}
\label{thm:14contradicts-conj-hyp}
If there exist $\varepsilon_0>0$ and $N_0>0$ such that, for all $n>N_0$, 
the number of non-$\frac14$-regular knots of~$n$ or fewer crossings 
is at least~$\varepsilon_0 H_n$, where $H_n$ is the number of hyperbolic knots of~$n$ or fewer crossings, 
then Conjecture~\ref{H1} does not hold.
\end{thm}

\begin{proof}
Suppose that the assumption of the theorem holds true,
denote by $\mathscr{M}_{\frac14}$ the set of all non-$\frac14$-regular knots,
and let $f$ be a map with domain $\mathscr{M}_{\frac14}$ sending 
$K\in\mathscr{M}_{\frac14}$ to 
a composite knot $f(K)$ with factor $K$ such that  
\be
\label{eq:f-def}
\crn(f(K))~<~ \frac14 \crn(K).
\ee
Then the result of Lackenby~\cite{La09} stating that for any knots $K_1$, $\dots$, $K_n$ in the $3$-sphere we have
\be
\label{eq:Lackenby152}
\frac{\crn(K_1)+\dots+\crn(K_n)}{152}~\le~\crn(K_1\sharp\dots\sharp K_n)
\ee
implies that for each knot $L$ in the codomain~$f(\mathscr{M}_{\frac14})$ we have
\be
\label{eq:38}
\card(f^{-1}(L))~<~152/4=38.
\ee

Indeed, let $K$ be a knot with $f(K)=L$ 
having the smallest crossing number 
among the elements of $f^{-1}(L)$.  
Then~\eqref{eq:Lackenby152} implies that
\be
\label{eq:Lack-cor}
\frac{\card(f^{-1}(L))\crn(K)}{152}~\le~\crn(L).
\ee
Obviously, \eqref{eq:f-def} and~\eqref{eq:Lack-cor} imply~\eqref{eq:38}.

Since all of the knots in~$f(\mathscr{M}_{\frac14})$ are composite, 
it follows by~\eqref{eq:f-def} and~\eqref{eq:38}
that
for all $n\in\mathbb{N}$ we have
\be
\label{eq:composite-1}
\frac{\card\{K\in\mathscr{M}_{\frac14} : \crn(K) \le 4n\}}{38}~<~\card\{L\in f(\mathscr{M}_{\frac14}) : \crn(L) \le n\}~\le~C_n,
\ee
where $C_n$ is the number of composite knots of~$n$ or fewer crossings.
At the other hand, by the assumption of the theorem, for all $m>N_0$ we have
\be
\label{eq:composite-2}
H_m\varepsilon_0~\le~\card\{K\in\mathscr{M}_{\frac14} : \crn(K) \le m\}.
\ee
Then~\eqref{eq:composite-1} and~\eqref{eq:composite-2} imply that
for all $n> N_0/4$ we have
\be
\label{eq:composite-3}
\frac{\varepsilon_0}{38} H_{4n}~<~C_n.
\ee

Now, we observe that each knot~$K$ in the $3$-sphere obviously has a two-strand cable knot~$J_K$ with $\crn(J_K)\le 4\crn(K)+1$. 
Since a cable knot over a nontrivial knot is a prime satellite knot (see \cite[p. 250, Satz 4]{Schu53}, \cite[Cor.~2]{Gra91}),
while cable knots over distinct knots are distinct (Lemma~\ref{lem:cables-are-distinct} below), it follows by~\eqref{eq:composite-3} that 
for all $n> N_0/4$ we have
$$
\frac{\varepsilon_0}{38}H_{4n}~<~C_n~<~S_{4n+1},
$$
where $S_m$ denotes the 
number of all prime satellite knots of~$m$ or fewer crossings.
Consequently, since the sequences $(H_i)_{i\in\mathbb{N}}$ and $(S_i)_{i\in\mathbb{N}}$ are monotonically increasing, for all $m> N_0$ we have
$$
\frac{\varepsilon_0}{38}H_{m}~<~S_{m+4}.
$$
%
As is shown in Section~\ref{sec:idea} 
(see deduction of Theorem~{\rm\ref{T1}} from Proposition~{\rm\ref{lem:main}}),
conditions of this kind contradict Conjecture~\ref{H1}.
\end{proof}

\begin{lem}
\label{lem:cables-are-distinct}
Cable knots over distinct knots are distinct.
\end{lem}

\begin{proof}
By Corollary~2 of~\cite{FW78}, the group of a cable knot $J(p, q; K)$ determines the numbers $|p|$ and $|q|$ and the topological type of $K$'s complement. 
By the Gordon--Luecke theorem~\cite{GL89}, the knot complement determines the knot.
\end{proof}

\section{Addendum III: Weak property~PT and unknotting numbers}
\label{sec:add-III}

This section deals with a relation between weak property~PT and 
the unknotting number of knots.
The unknotting number of a knot~$K$ is denoted by $u(K)$.

\subsection*{Definitions}
Let us say that a knot $P$ is \emph{weakly U-regular} if we have $u(P)\le u(K)$ whenever $P$ is a factor of a knot~$K$.
We say that a knot~$P$ is \emph{strictly U-regular} if we have $u(P)< u(K)$ whenever $P$ is a factor of a knot~$K\neq P$.
We say that a knot~$P$ has \emph{weak BJ-property} if 
by altering one of the crossings in a minimal diagram of~$P$ we obtain a knot $J\neq P$ with $u(J)\le u(P)$. 
We say that a knot~$P$ has \emph{strict BJ-property} if 
by altering one of the crossings in a minimal diagram of~$P$ we obtain a knot $J$ with $u(J)<u(P)$. 

\subsection*{Remarks}
1. The conjecture that all knots are strictly U-regular is weaker than 
the old conjecture on additivity of the unknotting number of knots under connected sum
(see, e.\,g.,~\cite[p.\,61]{Ad94b}, \cite[Problem 1.69]{Kir97}).
At the moment, no counterexample seems to be known to the latter conjecture.
Thus, no examples of non-U-regular knots are known up to now.
The theorem of Scharlemann~\cite{Scha85} saying that unknotting number one knots are prime
(together with the Unique Factorization Theorem by Schubert~\cite{Schu49})
implies that 
all knots with unknotting number one are strictly U-regular, while
all knots with unknotting number two are weakly U-regular.


2. The so-called Bernhard--Jablan conjecture (see \cite{Be94}, \cite{Ja98}, and \cite{JS07}) is equivalent to the conjecture that all knots have strict BJ-property.
Kohn's conjecture \cite[Conjecture~12]{Koh91} (which can be viewed as a particular case of the Bernhard--Jablan conjecture) is equivalent to the conjecture that all knots with unknotting number one have strict BJ-property.
The set of knots with strict BJ-property contains the set of knots satisfying the Bernhard--Jablan conjecture. 
At the moment, no counterexample seems to be known to the Bernhard--Jablan conjecture.
Available results concerning unknotting number shows that many small knots and some specific classes of knots satisfy the Bernhard--Jablan conjecture, hence have strict BJ-property.
For example, results of~\cite{KrM93} and~\cite{Mur91} imply that all torus knots have strict BJ-property.
Results of McCoy~\cite{McC13} imply that alternating knots with unknotting number one have strict  BJ-property.

\begin{prop}
\label{prop:BJimpliesPT}
1. Each \underline{weakly} U-regular prime knot with \underline{strict} BJ-property has weak property~PT.

2. Each \underline{strictly} U-regular prime knot with \underline{weak} BJ-property has weak property~PT.
\end{prop}

\begin{proof}
If $P$ is a weakly [resp., strictly] U-regular prime knot with strict [resp., weak] BJ-property, then
there exists a minimal diagram~$D_P$ of~$P$ (on the $2$-sphere $S^2=\R^2\cup\{\infty\}$) with a crossing~$X_1$ such that the change of the crossing yields a diagram of a knot~$J$ with $u(J)=u(P)-1$ [resp., a knot~$J\neq P$ with $u(J)\le u(P)$]. 
Let $d$ be a disk in~$S^2$ containing~$x_1$ 
such that the intersection $d \cap D_P$ is homeomorphic to~$\times$
while $\dd d$ intersects $D_P$ transversally in four points.
Let $\delta$ denote the disk $S^2\sm \inr(d)$, and let
$(B,t)$ be the $2$-string tangle represented by the diagram $\delta \cap D_P$.

We show that $(B,t)$ has no local knots.
Suppose on the contrary that $(B,t)$ is locally knotted, that is,
$B$ contains a ball $A$ such that the pair $(A,A\cap t)$ is a nontrivial $1$-string tangle.
We denote by~$L$ the knot that is the closure of the $1$-string tangle ${(A,A\cap t)}$.
Then $L$ is a factor of~$P$.
Since $P$ is prime and $L$ is nontrivial, it follows that $L$ and $P$ are equivalent.
At the same time, $L(=P)$ is a factor of~$J$.
Then we have $u(P)\le u(J)$ because $P$ is weakly U-regular
[resp., $u(P)< u(J)$ because $P$ is strictly U-regular while $J\neq P$].
However, $u(J)=u(P)-1$ [resp., $u(J)\le u(P)$]. 
The~obtained contradiction proves that $(B,t)$ has no local knots.

Now, we take a subdisk $d'$ in~$d$ such that the intersection $d' \cap D_P$ consists of 
two subarcs on two distinct legs of $\times=d \cap D_P$ 
(while $\dd d'$ intersects $D_P$ transversely in four points):

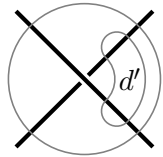
\begin{figure}[ht]
\begin{center}
\begin{tikzpicture} 
[knots/.style={white,double=black,very thick,double distance=1.6pt}]
\draw[knots](-0.9,-0.9)--(0.9,0.9);
\draw[knots](-0.9,0.9)--(0.9,-0.9);
\draw[help lines, line width=0.6pt](0,0) circle (5/5);
\draw[help lines, line width=0.6pt] (-45:4/5) arc (-45:45:4/5) arc (45:225:1/5) arc (45:-45:2/5) arc (135:315:1/5);
\draw (3/5,0) node {$d'$};
\end{tikzpicture} 
\end{center}
\end{figure}

Let $\delta'$ denote the disk $S^2\sm \inr(d')$.
Obviously, the diagram $\delta' \cap D_P$ represents the same 
$2$-string tangle $(B,t)$, which has no local knots.
Thus, the requirements from the definition of weak property~PT are fulfilled.
Consequently, $P$ has weak property~PT.
\end{proof}

\begin{cor}
\label{cor:BJcontradicts-1}
If there exist $\varepsilon>0$ and $N>0$ such that, for all $n>N$, the percentage of 
\underline{weakly} U-regular knots with \underline{strict} BJ-property
amongst all of the hyperbolic knots of~$n$ or fewer crossings is at least~$\varepsilon$, then
Conjecture~\ref{H1} does not hold.
\end{cor}

\begin{proof}
By Proposition~\ref{prop:BJimpliesPT}, 
the assumption of the corollary implies Conjecture~\ref{conj-weak-PT} 
(which concerns the set of knots having weak property~PT).
By Theorem~\ref{T2}, Conjecture~\ref{conj-weak-PT} contradicts Conjecture~\ref{H1}.
\end{proof}

\begin{cor}
\label{cor:BJcontradicts-2}
If there exist $\varepsilon>0$ and $N>0$ such that, for all $n>N$, the percentage of 
\underline{strictly} U-regular knots with \underline{weak} BJ-property
amongst all of the hyperbolic knots of~$n$ or fewer crossings is at least~$\varepsilon$, then
Conjecture~\ref{H1} does not hold.
\end{cor}

\begin{proof}
See the proof of Corollary~\ref{cor:BJcontradicts-1}.
\end{proof}

\end{document}